\documentclass[11 pt, oneside]{amsart}
\usepackage[top=1.5in,bottom=1.1in,left=1in,right=1.1in]{geometry}
\usepackage{amsmath,amssymb,amsthm,amsopn,extarrows}  

\usepackage[linktocpage=true]{hyperref}
\usepackage[dvipsnames]{xcolor}

\hypersetup{
    colorlinks = true,
    allcolors = {blue}}
\usepackage[dvipsnames]{xcolor}

\usepackage{setspace}
\usepackage{tikz-cd,tikz}
\usetikzlibrary{matrix,arrows,decorations.pathmorphing}

\newtheorem{lemma}[equation]{Lemma}
\newtheorem{theorem}[equation]{Theorem}
\newtheorem{proposition}[equation]{Proposition}
\newtheorem{corollary}[equation]{Corollary}
\newtheorem{theoremp}{Theorem}

\newtheorem{corollary*}[theoremp]{Corollary}
\newtheorem{proposition*}[theoremp]{Proposition}
\theoremstyle{definition}
\newtheorem{definition}[equation]{Definition}
\newtheorem*{remark*}{Remark}
\newtheorem{remark}[equation]{Remark}
\newtheorem*{example}{Examples}
\newtheorem*{properties}{General properties of Hessenberg varieties}

\numberwithin{equation}{section}

\makeatletter
\@namedef{subjclassname@1991}{MSC2020}
\makeatother

\DeclareMathOperator{\Hess}{Hess}
\DeclareMathOperator{\pt}{pt}
\DeclareMathOperator{\codim}{codim}

\newcommand{\hooklongrightarrow}{\lhook\joinrel\longrightarrow}

\title{Perverse sheaves and the cohomology of regular Hessenberg varieties}
\author[Ana B\u{a}libanu]{Ana B\u{a}libanu}
\address[Ana B\u{a}libanu]{Department of Mathematics, Harvard University, 1 Oxford Street, Cambridge, MA  02138, USA}
\email{ana@math.harvard.edu}
\author[Peter Crooks]{Peter Crooks}
\address[Peter Crooks]{Department of Mathematics, Northeastern University, 360 Huntington Avenue, Boston, MA  02115, USA}
\email{p.crooks@northeastern.edu}
\date{}\date{}

\begin{document}
\maketitle

\begin{abstract}
We use the Springer correspondence to give a partial characterization of the irreducible representations which appear in the Tymoczko dot action of the Weyl group on the cohomology ring of a regular semisimple Hessenberg variety. In type $A$, we apply these techniques to prove that all irreducible summands which appear in the pushforward of the constant sheaf on the universal Hessenberg family have full support. 

We also observe that the recent results of Brosnan and Chow, which apply the local invariant cycle theorem to the family of regular Hessenberg varieties in type $A$, extend to arbitrary Lie type. We use this extension to prove that regular Hessenberg varieties, though not necessarily smooth, always have the ``K\"ahler package.''
\end{abstract}

\setcounter{tocdepth}{1}
\tableofcontents

\section*{Introduction}
Let $G$ be a connected, simply-connected, semisimple complex algebraic group with Lie algebra $\mathfrak{g}$. Fix a Borel subgroup $B$ with Lie algebra $\mathfrak{b}$, and a $B$-stable subspace $H$ of $\mathfrak{g}$ which contains $\mathfrak{b}$. The Hessenberg variety associated to an element $x$ of $\mathfrak{g}$ is
\[\Hess(x,H)=\left\{gB\in G/B\mid g^{-1}\cdot x\in H\right\}.\]
This variety is the fiber above $x$ of a Poisson moment map 
\begin{align*}
\mu_H: G\times^B H&\longrightarrow\mathfrak{g} \\
			[g:y]&\longmapsto g\cdot y.
			\end{align*} 
In this way, the family of Hessenberg varieties is a generalization of the Grothendieck--Springer simultaneous resolution, which corresponds to the case $H=\mathfrak{b}$.

The study of Hessenberg varieties lies at the intersection of algebraic geometry, representation theory, and combinatorics. Examples of these varieties first appeared in applications to numerical analysis due to De Mari and Shayman \cite{dem.sha:88}. They were then defined in full generality by De Mari, Procesi, and Shayman in \cite{dem.pro.sha:92}, who described the geometry of Hessenberg varieties corresponding to regular semisimple elements. In this work the authors showed that such varieties are smooth, and that one particular case is the toric variety whose fan is given by the Weyl chambers. In the same period, a singular Hessenberg variety known as the Peterson variety was introduced by Peterson in unpublished work. It came to play a central role in the study of the quantum cohomology rings of flag varieties in work of Kostant \cite{kos:96} and Rietsch \cite{rie:01, rie:03}. More recently, Goresky, Kottwitz, and MacPherson \cite{gor.kot.mac:06} have shown that affine Springer fibers admit pavings by affine bundles over generalized Hessenberg varieties.

The topology of Hessenberg varieties in type $A$ has been studied in detail by Tymoczko. In \cite{tym:08} she observed that regular semisimple Hessenberg varieties are GKM varieties, and she showed that their singular cohomology rings carry an action of the symmetric group called the Tymoczko dot action. Subsequently, Shareshian and Wachs \cite{sha.wac:16} conjectured a relationship between the character of the dot action in type $A$ and a generalization of the chromatic symmetric functions introduced by Stanley \cite{sta:95}. In particular, this conjecture gives an explicit decomposition of Tymoczko's action into a sum of irreducible representations of the symmetric group, building on earlier work of Gasharov \cite{gas:96}. The Shareshian--Wachs conjecture was recently proved by Brosnan and Chow \cite{bro.cho:18} by applying tools from the formalism of derived categories. The key insight of their approach is that, on the regular semisimple locus of the Hessenberg family, the dot action is induced by monodromy. 

In general Lie type, the dot action becomes an action of the Weyl group $W$. The problem of determining its character is still open. Prompted by the approach of Brosnan and Chow, we give the following result in this direction. Recall that the classical Springer correspondence assigns to each irreducible representation $\psi$ of $W$ a pair $(\mathcal{O}_\psi,\mathcal{L}_\psi)$ of a nilpotent orbit $\mathcal{O}_\psi$ and an irreducible local system $\mathcal{L}_\psi$ on $\mathcal{O}_\psi$. 

\begin{theoremp}
\label{first}
Let $H^\perp\subset\mathfrak{g}$ be the annihilator of $H$ with respect to the Killing form. Suppose that $\psi$ is an irreducible representation of $W$ which appears as a subrepresentation of the action of $W$ on the cohomology of a regular semisimple Hessenberg variety associated to $H$. Then the intersection $\mathcal{O}_\psi\cap H^\perp$ is nonempty.
\end{theoremp}  

In type $A$, the Springer correspondence identifies irreducible representations of the symmetric group $S_n$ with nilpotent adjoint orbits in $\mathfrak{sl}_n$. Both of these sets are indexed by partitions of the positive integer $n$. In this case, Theorem \ref{first} has the following more concrete statement, which can be deduced from the results of Brosnan and Chow and appears to be known to experts.

\begin{corollary*}
Suppose that $G=SL_n$, and let $\lambda$ be a partition of $n$ corresponding to an irreducible representation $\psi_{\lambda}$ of $S_n$. Suppose that $\psi_{\lambda}$ appears as a subrepresentation of the action of $W$ on the cohomology of a regular semisimple Hessenberg variety associated to $H$. Then the annihilator $H^\perp$ contains a nilpotent element whose Jordan normal form is given by the conjugate partition $\lambda'$.
\end{corollary*}

Using similar reasoning, in type $A$ we also prove a support theorem for the Hessenberg family $\mu_H:G\times^BH\longrightarrow\mathfrak{g}$. We are motivated by the recent work of Chen, Vilonen, and Xue \cite{che.vil.xue:17, che.vil.xue:18, che.vil.xue:19, che.vil.xue:20}, in which the authors apply the Fourier transform to study the singular cohomology of several classes of algebraic varieties which are related to Hessenberg varieties.

\begin{theoremp}
\label{second}
Suppose that $G=SL_n$. In the bounded derived category of constructible complexes of sheaves on $\mathfrak{g}$, all irreducible summands of the derived pushfoward $\mu_{H*}\underline{\mathbb{C}}_{G\times^BH}$ have full support.
\end{theoremp}

Precup has shown that any regular Hessenberg variety admits an affine paving \cite{pre:13} and has palindromic Betti numbers \cite{pre:17}. This makes it possible to extend a result of Brosnan and Chow \cite[Theorem 127]{bro.cho:18}---which identifies the singular cohomology of any regular Hessenberg variety with an invariant subring of the cohomology of a regular semisimple Hessenberg variety---from type $A$ to general Lie type. We apply this to show that regular Hessenberg varieties, though not necessarily smooth, have the ``K\"ahler package.'' This extends certain results of \cite{abe.hor.mas.mur.sat:19} in the regular nilpotent case, and gives additional evidence for the conjecture of Precup \cite{pre:17} that regular Hessenberg varieties are rationally smooth.

\begin{theoremp}
\label{third}
For any regular element $x$ of $\mathfrak{g}$, the singular cohomology ring $H^*(\Hess(x,H))$ satisfies Poincar\'e duality, the hard Lefschetz property, and the Hodge--Riemann relations.
\end{theoremp}

In Section \ref{monodromy section} we review the construction of Hessenberg varieties, the definition of the dot action, and its interpretation in terms of monodromy. In Section \ref{decomp section} we recall some conventions on the decomposition theorem, the Fourier transform, and the Springer correspondence. In Section \ref{what appears} we use these tools to prove Theorems \ref{first} and \ref{second}, which appear as Theorem \ref{summands} and Theorem \ref{main}. In Section \ref{locinv} we show that the results of Brosnan and Chow extend to all Lie types, and we apply them to prove Theorem \ref{third}, as Theorem \ref{kahler}. We include in an appendix some technical results about monodromy actions on equivariant cohomology. 

We are grateful to Victor Ginzburg for making us aware of the work of Chen, Vilonen, and Xue, and to Martha Precup for interesting discussions. We also thank the anonymous referees for their detailed and constructive comments. During the completion of this work, A.B. was partially supported by a National Science Foundation MSPRF under award DMS--1902921, and P.C. was partially supported by an NSERC Postdoctoral Fellowship under award PDF--516638.\\

%
%
%
%
%
%
%
%
\section{Monodromy actions of the Weyl group}
\label{monodromy section}
\subsection{Recollections on Hessenberg varieties}
Let $G$ be a connected, simply-connected, semisimple algebraic group over $\mathbb{C}$ and let $\mathfrak{g}$ be its Lie algebra. Fix a maximal torus $T$ and a Borel subgroup $B$ containing it, and write $\mathfrak{b}$ for the Lie algebra of $B$. Let $\mathcal{B}$ be the flag variety of all Borel subalgebras of $\mathfrak{g}$, which we freely identify with the homogeneous space $G/B$.

\begin{definition}
A \emph{Hessenberg subspace} of $\mathfrak{g}$ is a $B$-submodule $H\subset\mathfrak{g}$ that contains $\mathfrak{b}$.
\end{definition}

Given a Hessenberg subspace $H$, consider the associated $G$-equivariant vector bundle 
\[G\times^B H\longrightarrow G/B.\]
The total space of this vector bundle has a natural Poisson structure (introduced in \cite{abe.cro:18} and studied in \cite{bal:19}) for which the action of $G$ is Hamiltonian. The moment map is
\begin{align*}
\mu_H:\,G\times^B H&\longrightarrow \mathfrak{g} \\
	[g:x]&\longmapsto g\cdot x,
	\end{align*}
where we identify $\mathfrak{g}\cong\mathfrak{g}^*$ via the Killing form and where $g\cdot x$ denotes the adjoint action. 

\begin{remark}
\label{properness}
For any $B$-stable subspace $V\subseteq\mathfrak{g}$, the map
\begin{align*}
\mu_V:\,G\times^B V&\longrightarrow \mathfrak{g} \\
	[g:x]&\longmapsto g\cdot x
	\end{align*}
factors through the closed embedding
\begin{align*}
G\times^B V&\longrightarrow G/B\times \mathfrak{g} \\
	[g:x]&\longmapsto (gB,g\cdot x).
	\end{align*}
There is a commutative diagram
\begin{equation*}
\begin{tikzcd}[row sep=huge]
G\times^BV		\arrow[hookrightarrow]{r}\arrow{rd}[swap]{\mu_V}		&G/B\times \mathfrak{g}\arrow{d}{}\\
  														&\mathfrak{g},
\end{tikzcd}
\end{equation*}
where the vertical arrow is projection onto the second component. This implies that any morphism of the form $\mu_V$ is projective, and in particular proper.
\end{remark}

\begin{definition}
The \emph{Hessenberg variety} associated to the subspace $H$ and to a point $x\in\mathfrak{g}$ is
\begin{align*}
\Hess(x,H)&=\mu_H^{-1}(x) \\
		&=\left\{gB\in G/B\mid g^{-1}\cdot x\in H\right\}.
		\end{align*}
We call this Hessenberg variety \emph{regular} (resp. \emph{semisimple, nilpotent}) if $x$ is a regular (resp. semisimple, nilpotent) element of $\mathfrak{g}$.
\end{definition}

\begin{example}
(1) When $H=\mathfrak{b}$, the vector bundle $\tilde{\mathfrak{g}}=G\times^B\mathfrak{b}$ is the total space of the Grothendieck--Springer simultaneous resolution. Identifying the homogeneous space $G/B$ with the flag variety, the Hessenberg variety $\Hess(x,\mathfrak{b})$ is precisely the Grothendieck--Springer fiber 
\[\mathcal{B}_x=\{\mathfrak{b}'\in \mathcal{B}\mid x\in \mathfrak{b}'\}.\]

(2) Let $\Delta$ be the set of simple roots determined by $T$ and $B$, and consider the \emph{standard Hessenberg subspace}
\[H_0=\left(\sum_{\alpha\in\Delta}\mathfrak{g}_{-\alpha}\right)\oplus\mathfrak{b}.\]
For any regular element $x\in\mathfrak{g}$, the centralizer $G_x=\{g\in G\mid g\cdot x=x\}$ acts on the corresponding Hessenberg variety $\Hess(x,H_0)$ with an open dense orbit \cite{bal:19}. When $s\in\mathfrak{g}$ is regular and semisimple, $\Hess(s,H_0)$ is the toric variety corresponding to the fan of Weyl chambers \cite[Theorem 11]{dem.pro.sha:92}. When $e\in\mathfrak{g}$ is regular and nilpotent, $\Hess(e,H_0)$ is the Peterson variety \cite{rie:03}.
\end{example}

\begin{remark}
\label{fibers}
For any Hessenberg subspace $H$, there is a commutative diagram
\begin{equation*}
\begin{tikzcd}[row sep=huge]
\tilde{\mathfrak{g}}=G\times^B\mathfrak{b}		\arrow[hookrightarrow]{r}\arrow{rd}[swap]{\mu_{\mathfrak{b}}}		&G\times^B H\arrow{d}{\mu_H}\\
  														&\mathfrak{g}.
\end{tikzcd}
\end{equation*}
For every $x\in\mathfrak{g}$, this gives a natural $G_x$-equivariant inclusion 
\[\mathcal{B}_x\hooklongrightarrow\Hess(x,H).\]
When $x\in\mathfrak{g}$ is regular, the Grothendieck--Springer fiber $\mathcal{B}_x$ is precisely the set of fixed points for the action of $G_x$ on $\mathcal{B}$. It follows that for all regular $x\in\mathfrak{g}$,
\[\mathcal{B}_x=\Hess(x,H)^{G_x}.\]
\end{remark}

We recall some features of the geometry of Hessenberg varieties. Let $\mathfrak{g}^{\text{r}}$ (resp. $\mathfrak{g}^{\text{rs}}$) denote the locus of regular (resp. regular semisimple) elements of $\mathfrak{g}$. Given any subspace $V\subset\mathfrak{g}$, we write $V^{\text{r}}$ for the regular locus $V\cap\mathfrak{g}^{\text{r}}$ and $V^{\text{rs}}$ for the regular semisimple locus $V\cap \mathfrak{g}^{\text{rs}}$.  If $X$ is a topological space, we denote by $H^*(X)$ the singular cohomology of $X$ with complex coefficients. 

\begin{properties}
(1) For any regular element $x\in\mathfrak{g}$, the Hessenberg variety $\Hess(x,H)$ has dimension equal to $\dim(H/\mathfrak{b})$ \cite[Corollary 3]{pre:17}.

(2) When $s\in\mathfrak{g}$ is regular and semisimple, the Hessenberg variety $\Hess(s,H)$ is smooth. The restriction  
\[\mu_H:G\times^BH^{\text{rs}}\longrightarrow\mathfrak{g}^{\text{rs}}\]
of $\mu_H$ to the regular semisimple locus is a smooth morphism \cite[Theorem 6]{dem.pro.sha:92}.

(3) A Hessenberg subspace $H$ is called \emph{indecomposable} if it contains every negative simple root space. If $H$ is indecomposable, then $\Hess(s,H)$ is connected for all semisimple $s\in\mathfrak{g}$ \cite{pre:15}. It follows from a Zariski Main Theorem argument \cite[Remark 4.6]{bal:19} that all Hessenberg varieties $\Hess(x,H)$ associated to an indecomposable $H$ are connected.

(4) It is shown in \cite{pre:17} that for any regular $x\in\mathfrak{g}$, the Hessenberg variety $\Hess(x,H)$ has palindromic Betti numbers. In other words,
\[\dim H^k(\Hess(x,H))=\dim H^{\text{top}-k}(\Hess(x,H))\qquad\text{for any }k\in\mathbb{Z},\]
where $\text{top}=2\dim (H/\mathfrak{b})$. When $H$ is indecomposable and $x\in\mathfrak{g}$ is regular, connectedness and palindromicity imply that the top cohomology group $H^{\text{top}}(\Hess(x,H))$ has dimension 1. It follows that all regular Hessenberg varieties $\Hess(x,H)$ associated to an indecomposable $H$ are irreducible \cite[Corollary 14]{pre:17}.

(5) It is proved in \cite{abe.ded.gal.har:18} and \cite{abe.fuj.zen:20} that when $H$ is indecomposable, any regular Hessenberg variety is reduced. When $H$ is not indecomposable, this is not necessarily the case---for instance, if $e\in\mathfrak{g}$ is regular and nilpotent, the Grothendieck--Springer fiber $\mathcal{B}_e=\mu_{\mathfrak{b}}^{-1}(e)$ is not reduced \cite[Section 1.3.2]{yun:17}. Since we are only concerned with the topology of Hessenberg varieties, the potentially non-reduced structure will not be relevant.
\end{properties}

%
%
%
%
%
%
%
%
%
\subsection{GKM varieties and monodromy}
Let $A$ be a complex torus. A smooth projective $A$-variety $X$ is called a \emph{GKM variety} if 
\begin{itemize}
\item the set of $A$-fixed points on $X$ is finite, and
\item the set of one-dimensional $A$-orbits on $X$ is finite.
\end{itemize}

\begin{remark}
An arbitrary projective $A$-variety is GKM if in addition to these conditions it is equivariantly formal in the sense of \cite{gor.kot.mac:98}. Since smoothness implies equivariant formality, and since all of our GKM varieties will be smooth, we use the definition above for simplicity.
\end{remark}

We write $H^*_A(X)$ for the $A$-equivariant cohomology of $X$ with complex coefficients. In the rest of this section we will use a number of standard facts about equivariant cohomology, which we review in the first section of the appendix. 

\begin{proposition}\cite[Theorem 1.6.2]{gor.kot.mac:98}
\label{gkm}
 Suppose that $X$ is a GKM variety. 
\begin{enumerate}
\item The restriction map $H^*_A(X)\longrightarrow H^*_A(X^A)$ is injective.
\item The specialization map $H^*_A(X)\longrightarrow H^*(X)$ is surjective.
\end{enumerate}
\end{proposition}

Let $H \subset\mathfrak{g}$ be a Hessenberg subspace and let $s\in\mathfrak{g}$ be a regular semisimple element. The centralizer $G_s$ is a maximal torus of $G$ which acts on the associated Hessenberg variety $\Hess(s,H)$.

\begin{proposition} \cite[Section III]{dem.pro.sha:92}
\label{hessgkm}
The Hessenberg variety $\Hess(s,H)$ is a GKM variety for the action of $G_s$.
\end{proposition}

\begin{remark}
Using Proposition \ref{gkm}, Tymoczko \cite{tym:08} defined an action of the Weyl group on $H^*(\Hess(s,H))$. In the rest of this section we will show that it is induced by the natural monodromy action of the fundamental group $\pi_1(\mathfrak{g}^{\text{rs}},s)$. 

The connection between the dot action and monodromy was established by Brosnan and Chow. While their paper \cite{bro.cho:18} only makes this identification in type $A$, there is a straightforward generalization of their argument to arbitrary Lie type. We formulate it below for completeness, and we also include in the appendix a number of details about monodromy actions on equivariant cohomology.
\end{remark}

Let $\mathfrak{t}$ be the Lie algebra of the maximal torus $T$, and let $W$ be the associated Weyl group. Consider the restriction $\tilde{\mathfrak{g}}^{\text{rs}}=\mu_{\mathfrak{b}}^{-1}(\mathfrak{g}^{\text{rs}})$ of the Grothendieck--Springer resolution to the regular semisimple locus. There is a Cartesian diagram
\begin{equation}
\label{rs-groth-springer}
\begin{tikzcd}[row sep=huge]
\tilde{\mathfrak{g}}^{\text{rs}}=G\times^B\mathfrak{b}^{\text{rs}}		\arrow{r}\arrow{d}{\mu_{\mathfrak{b}}}	&\mathfrak{t^{\text{r}}}	\arrow{d}\\
\mathfrak{g}^{\text{rs}}	\arrow{r}					&\mathfrak{t}^{\text{r}}/W,
\end{tikzcd}
\end{equation}
where the top horizontal arrow is projection onto the first summand of the decomposition $\mathfrak{b}=\mathfrak{t}\oplus[\mathfrak{b},\mathfrak{b}]$, and the bottom arrow is induced by the Chevalley isomorphism \cite[Theorem 3.1.38]{chr.gin:97}. The smooth morphism
\[\tilde{\mathfrak{g}}^{\text{rs}}\longrightarrow\mathfrak{g}^{\text{rs}}\]
is a Galois cover with Galois group  $W$ \cite[Proposition 3.1.36]{chr.gin:97}. Therefore, for every $s\in\mathfrak{g}^{\text{rs}}$ with fixed preimage $\tilde{s}\in\tilde{\mathfrak{g}}^{\text{rs}}$, there is a surjective group homomorphism
\begin{equation}
\label{factors}
\tag{$\star$}
\pi_1(\mathfrak{g}^{\text{rs}},s)\longrightarrow W
\end{equation}
whose kernel is the image of the natural map
\[\pi_1(\tilde{\mathfrak{g}}^{\text{rs}},\tilde{s})\longrightarrow\pi_1(\mathfrak{g}^{\text{rs}},s).\]

For every $s\in\mathfrak{t}^{\text{r}}$, the morphism
\[\mu_H: G\times^BH^{\text{rs}}\longrightarrow \mathfrak{g}^{\text{rs}}\]
induces monodromy actions of $\pi_1(\mathfrak{g}^{\text{rs}},s)$ on $H^*(\Hess(s,H))$ and on $H^*_T(\Hess(s,H))$, as explained in \cite[Chapter I]{del:70}. These actions are described in the following proposition, whose proof we postpone to the appendix.

\begin{proposition}
\label{monodromy factors}
Let $s\in\mathfrak{t}$ be a regular element. The monodromy action of $\pi_1(\mathfrak{g}^{\text{\emph{rs}}},s)$ on $H^*(\Hess(s,H))$ factors through \eqref{factors}.
\end{proposition}

%
%
%
%
%
%
%
%
%
\subsection{The Tymoczko dot action}
There is a natural isomorphism $H_T(\pt)\cong\mathbb{C}[\mathfrak{t}]$ \cite[Example 1.2]{bri:98} between the $T$-equivariant cohomology of a point and the algebra of polynomial functions on $\mathfrak{t}$. For any regular element $s\in\mathfrak{t}$ this gives an isomorphism
\[H^*_T(\mathcal{B}_s)\cong\bigoplus_{\tilde{s}\in\mathcal{B}_s}\mathbb{C}[\mathfrak{t}],\]
where the left-hand side is the equivariant cohomology of the Grothendieck--Springer fiber $\mathcal{B}_s$, on which the Weyl group acts freely and transitively. The \emph{dot action} on $H^*_T(\mathcal{B}_s)$ is given by
\[w\cdot (f_{\tilde{s}})=(wf_{w^{-1}\tilde{s}})\qquad\text{for any }w\in W,\, (f_{\tilde{s}})\in \bigoplus_{\tilde{s}\in\mathcal{B}_s}\mathbb{C}[\mathfrak{t}],\]
where the action of $W$ on $\mathbb{C}[\mathfrak{t}]$ is induced by its action on $\mathfrak{t}$. 

Tymoczko originally defined this action in type $A$ \cite[Section 3.1]{tym:08}. She proved that it restricts to an action of $W$ on the image of the embedding
\begin{equation}
\label{restr}
H^*_T(\Hess(s,H))\hooklongrightarrow H^*_T(\mathcal{B}_s),
\end{equation}
and that it descends to an action of $W$ on $H^*(\Hess(s,H))$ through the surjection
\begin{equation}
\label{spec}
H^*_T(\Hess(s,H))\longrightarrow H^*(\Hess(s,H)).
\end{equation}
These results extend to arbitrary semisimple Lie algebras \cite[Section 8.3]{abe.hor.mas.mur.sat:19} to give a dot action of $W$ on the cohomology of any regular semisimple Hessenberg variety. 

We will show that the dot action on $H^*_T(\mathcal{B}_s)$ agrees with the monodromy action of $W$ induced by \eqref{factors}. This will imply that the dot action on $H^*(\Hess(s,H))$ coincides with the monodromy action of $W$ coming from Proposition \ref{monodromy factors}.

\begin{proposition}
\label{equivariant coincidence}
Let $s\in\mathfrak{t}$ be a regular element. The monodromy action of $W$ on $H^*_T(\mathcal{B}_s)$ coincides with the dot action.
\end{proposition}
\begin{proof}
There is an isomorphism of graded algebras
\[H^*_T(\mathcal{B}_s)\cong  H^0(\mathcal{B}_s)\otimes \mathbb{C}[\mathfrak{t}],\]
as in Remark \ref{weyl action} of the appendix. The dot action on $H^*_T(\mathcal{B}_s)$ is exactly the diagonal $W$-action induced by the natural actions of the Weyl group on $\mathcal{B}_s$ and on $\mathfrak{t}$. The conclusion now follows from Corollary \ref{factorscor}.
\end{proof}

\begin{corollary}
\label{coincidence}
Let $s\in\mathfrak{t}$ be a regular element. The monodromy action of $W$ on $H^*(\Hess(s,H))$ coincides with the dot action.
\end{corollary}
\begin{proof}
The dot action on $H^*(\Hess(s,H))$ is induced from the dot action of $W$ on $H^*_T(\mathcal{B}_s)$ through the maps \eqref{restr} and \eqref{spec}. In the proof of Proposition \ref{monodromy factors}, it is shown that the monodromy action on $H^*(\Hess(s,H))$ is induced from the monodromy action of $W$ on $H^*_T(\mathcal{B}_s)$ in the same way. In view of Proposition \ref{equivariant coincidence}, the two actions agree.\qedhere\\
\end{proof}

%
%
%
%
%
%
%
%
%
\section{The decomposition theorem and Springer theory}
\label{decomp section}
\subsection{The decomposition theorem}
For any algebraic variety $X$, we denote by $D(X)$ the bounded derived category of complexes of sheaves on $X$ which are constructible with respect to a fixed stratification. We write $\underline{\mathbb{C}}_X$ for the constant sheaf on $X$ with coefficients in $\mathbb{C}$. 

For any local system $\mathcal{L}$ on a stratum $S$, write $\mathcal{L}[-]=\mathcal{L}[\dim S]$ for the shift of $\mathcal{L}$ which is perverse. Let $IC_S(\mathcal{L})$ be the unique perverse sheaf on $X$ which is supported on $\overline{S}$ and whose restriction to $S$ is $\mathcal{L}[-]$. Now suppose that $X$ is smooth and that $\varphi:X\longrightarrow Y$ is a proper morphism. Up to shift, proper base change gives a graded isomorphism
\[(R^*\varphi_*\underline{\mathbb{C}}_X[-])_y\cong H^{*}(\varphi^{-1}(y))\]
between the cohomology of the stalks of the derived pushforward $\varphi_*\underline{\mathbb{C}}_X$ and the singular cohomology of the fibers of $\varphi$ \cite[Fact 2.2.1]{dec:17}.

We will use the following version of the decomposition theorem of Beilinson, Bernstein, Deligne, and Gabber \cite{bei.ber.del:82}. For more details on this setting, we refer to \cite[Section 1.6]{dec:17}.

\begin{theorem}[Decomposition Theorem]
\label{bbdg}
Suppose that $X$ is a smooth algebraic variety and that $\varphi:X\longrightarrow Y$ is a proper morphism. Then there is an isomorphism in the derived category
\[\varphi_*\underline{\mathbb{C}}_X\cong\bigoplus_{S,b}IC_S(\mathcal{L}_{S,b})[-b],\]
taken over all strata $S\subset Y$ and over all integers $b$, and where each $\mathcal{L}_{S,b}$ is a semisimple local system on the stratum $S$.
\end{theorem}

Fixing a basepoint $s\in S$ gives an equivalence between the category of local systems on $S$ and the category of finite-dimensional representations of $\pi_1(S,s)$ \cite[Theorem 13.2.3]{ara:12}. We write $L_s$ for the fiber of the local system $\mathcal{L}$ at $s$.

Fix a stratification
\[\mathfrak{g}=\bigsqcup S\]
into smooth, locally closed subvarieties such that the restriction of the proper morphism $\mu_\mathfrak{b}:\tilde{\mathfrak{g}}\longrightarrow \mathfrak{g}$ to the preimage of any stratum is a locally trivial fibration \cite{ver:76}. Without loss of generality, we take $\mathfrak{g}^{\text{rs}}$ to be the open dense stratum. 

The morphism $\mu_\mathfrak{b}:\tilde{\mathfrak{g}}\longrightarrow\mathfrak{g}$ is small. (For details see \cite[Lecture I]{yun:17}.) Therefore only IC complexes with full support and no shift appear in the decomposition of Theorem \ref{bbdg}, which becomes  
\[\mu_{\mathfrak{b}*}\underline{\mathbb{C}}_{\tilde{\mathfrak{g}}}[-]\cong IC_{\mathfrak{g}^{\text{rs}}}(\mathcal{L}).\]
Since the monodromy action of $\pi_1(\mathfrak{g}^{\text{rs}},s)$ factors through \eqref{factors} and since the restriction of $\mu_{\mathfrak{b}}$ to the regular semisimple locus is a topological Galois cover for the action of the Weyl group, the semisimple local system $\mathcal{L}$ corresponds to the regular representation of $W$ \cite[Section 3.5]{dec:17}. Decomposing further into irreducible local systems, we obtain
\begin{equation}
\label{springer reps}
\mu_{\mathfrak{b}*}\underline{\mathbb{C}}_{\tilde{\mathfrak{g}}}[-]\cong\bigoplus IC_{\mathfrak{g}^{\text{rs}}}(\mathcal{L}_\psi)^{\oplus m_\psi}.
\end{equation}
Here the sum is taken over all irreducible representations $\psi$ of $W$, $\mathcal{L}_\psi$ is the local system corresponding to the irreducible representation $\psi$, and $m_\psi$ is the multiplicity of $\psi$ in the regular representation, which is equal to the dimension of $\psi$.  

\subsection{The Fourier transform}
Let $\mathcal{V}$ be a vector bundle over a smooth algebraic variety. A complex $\mathcal{F}\in D(\mathcal{V})$ is called \emph{monodromic} if its cohomology sheaves are locally constant along the orbits of the natural $\mathbb{C}^*$-action on $\mathcal{V}$. We write $D_{\text{mon}}(\mathcal{V})$ for the full subcategory of $D(\mathcal{V})$ consisting of monodromic complexes. 

There is a notion of Fourier transform for monodromic complexes \cite[Section 8]{gin:98} which gives a functor
\[\mathfrak{F}:D_{\text{mon}}(\mathcal{V})\longrightarrow D_{\text{mon}}(\mathcal{V}^*),\]
where $\mathcal{V}^*$ is the dual vector bundle. This functor induces an equivalence between the subcategories of monodromic perverse sheaves. In particular, if $\mathcal{W}$ is a subbundle of $\mathcal{V}$ and $\mathcal{W}^\perp\subset\mathcal{V}^*$ is its annihilator,
\begin{equation}
\label{perps}
\mathfrak{F}\left(\underline{\mathbb{C}}_{\mathcal{W}}[-]\right)\cong\underline{\mathbb{C}}_{\mathcal{W}^\perp}[-].
\end{equation}

Identifying $\mathfrak{g}\cong\mathfrak{g}^*$ via the Killing form, we obtain functors
\begin{align*}
&\mathfrak{F}:D_{\text{mon}}(G/B\times\mathfrak{g})\longrightarrow D_{\text{mon}}(G/B\times\mathfrak{g}),\\
&\mathfrak{F}:D_{\text{mon}}(\mathfrak{g})\longrightarrow D_{\text{mon}}(\mathfrak{g}).
\end{align*}
Recall that the Grothendieck--Springer resolution is a vector subbundle 
\begin{equation*}
\begin{tikzcd}[row sep=huge]
\tilde{\mathfrak{g}}=G\times^B\mathfrak{b}		\arrow[hookrightarrow]{r}\arrow{rd}		&G/B\times\mathfrak{g}	\arrow{d} \\
																		&G/B.
\end{tikzcd}
\end{equation*}
The annihilator of $\mathfrak{b}$ under the Killing form is the nilradical $\mathfrak{n}=[\mathfrak{b},\mathfrak{b}]$, and the annihilator of the vector bundle $\tilde{\mathfrak{g}}$ is the bundle
\[\tilde{\mathcal{N}}=G\times^B\mathfrak{n}\longrightarrow G/B.\]
There is a commutative diagram
\begin{equation*}
\begin{tikzcd}[row sep=huge]
\tilde{\mathfrak{g}}		\arrow[hookrightarrow]{r}\arrow{d}{\mu_\mathfrak{b}}		&G/B\times\mathfrak{g}	\arrow{d}{\mu}\arrow[hookleftarrow]{r} 	&\tilde{\mathcal{N}}	\arrow{d}{\mu_{\mathfrak{n}}}\\
\mathfrak{g}  		\arrow[equal]{r}								&\mathfrak{g}			\arrow[hookleftarrow]{r}				&\mathcal{N}.
\end{tikzcd}
\end{equation*}
Here $\mu$ is projection onto the second factor, $\mathcal{N}$ is the nilpotent cone of $\mathfrak{g}$, and $\mu_{\mathfrak{n}}$ is the Springer resolution. The morphism $\mu$ is $\mathbb{C}^*$-equivariant, so the derived pushforward gives a functor
\[\mu_{*}:D_{\text{mon}}(G/B\times\mathfrak{g})\longrightarrow D_{\text{mon}}(\mathfrak{g}).\]
The Fourier transform commutes with this functor \cite[Claim 8.4]{gin:98}, and equation \eqref{perps} implies that  
\begin{equation}
\label{fourier equivalence}
\mathfrak{F}\left(\mu_{\mathfrak{b}*}\underline{\mathbb{C}}_{\tilde{\mathfrak{g}}}[-]\right)\cong\mu_{\mathfrak{n}*}\mathfrak{F}\left(\underline{\mathbb{C}}_{\tilde{\mathfrak{g}}}[-]\right)\cong\mu_{\mathfrak{n}*}\underline{\mathbb{C}}_{\tilde{\mathcal{N}}}[-].
\end{equation}

The nilpotent cone is stratified by $G$-orbits into smooth, locally closed, $\mathbb{C}^*$-stable subvarieties along which the Springer morphism $\mu_\mathfrak{n}$ is locally trivial. Because $\mu_\mathfrak{n}$ is semismall, the derived pushforward $\mu_{\mathfrak{n}*}\underline{\mathbb{C}}_{\tilde{\mathcal{N}}}[-]$ is a perverse sheaf. (Once again we refer to \cite[Lecture I]{yun:17}.) The decomposition theorem gives an isomorphism
\[\mu_{\mathfrak{n}*}\underline{\mathbb{C}}_{\tilde{\mathcal{N}}}[-]\cong\bigoplus_{\mathcal{O}} IC_{\mathcal{O}}\left(\mathcal{M}_{\mathcal{O}}\right),\]
where each $\mathcal{O}\subset\mathcal{N}$ is a nilpotent $G$-orbit and $\mathcal{M}_\mathcal{O}$ is a semisimple local system on $\mathcal{O}$. 

Since $\mathcal{O}\cong G/G_e$, there is a long exact sequence of homotopy groups
\[\ldots\longrightarrow\pi_1(G)\longrightarrow\pi_1(\mathcal{O})\longrightarrow\pi_0(G_e)\longrightarrow\pi_0(G)\longrightarrow\ldots.\]
Because the group $G$ is connected and simply-connected, this sequence gives an isomorphism $\pi_1(\mathcal{O})\cong\pi_0(G_e)$. The fiber $M_{\mathcal{O},e}$ is therefore a representation of the group $\pi_0(G_e)$ of connected components of the centralizer of $e$.

By equations \eqref{springer reps} and \eqref{fourier equivalence},  
\[\mu_{\mathfrak{n}*}\underline{\mathbb{C}}_{\tilde{\mathcal{N}}}[-]\cong\bigoplus \mathfrak{F}(IC_{\mathfrak{g}^{\text{rs}}}(\mathcal{L}_\psi))^{\oplus m_\psi}.\]
For each index $\psi$, the Fourier transform $\mathfrak{F}(IC_{\mathfrak{g}^{\text{rs}}}(\mathcal{L}_\psi))$ is an irreducible perverse sheaf on $\mathcal{N}$. In other words, $\mathfrak{F}(IC_{\mathfrak{g}^{\text{rs}}}(\mathcal{L}_\psi))=IC_{\mathcal{O}_\psi}(\mathcal{M}_\psi)$ for some nilpotent orbit $\mathcal{O}_\psi$ equipped with an irreducible local system $\mathcal{M}_\psi$. 

The \emph{Springer correspondence} is the assignment
\begin{equation}
\label{actual corresp}
\psi\longmapsto  IC_{\mathcal{O}_\psi}(\mathcal{M}_\psi),
\end{equation}
which associates to each irreducible representation of the Weyl group this unique irreducible nilpotent orbital complex. 

\begin{remark}
\label{type A springer}
In type $A$, both irreducible representations of the symmetric group $S_n$ and nilpotent orbits in $\mathfrak{sl}_n$ are indexed by partitions of $n$. The Springer correspondence is a geometric realization of this bijection. 

Concretely, suppose that $G=SL_n$. The actions of $G$ on $\tilde{\mathcal{N}}$ and on $\mathcal{N}$ extend naturally to actions of $GL_n$, and the morphism
\[\mu_{\mathfrak{n}}:\tilde{\mathcal{N}}\longrightarrow\mathcal{N}\]
is $GL_n$-equivariant. Since the centralizer in $GL_n$ of any element in $\mathfrak{gl}_n$ is connected, the monodromy action of $\pi_1(\mathcal{O})\cong\pi_0(G_e)$ on $H^*\left(\mathcal{B}_e\right)$ is trivial. It follows that only sheaves $IC_{\mathcal{O}}(\mathcal{M})$ with $\mathcal{M}$ a trivial local system appear in the Springer correspondence in this case.

For any partition $\lambda$ of $n$, let $\lambda'$ be the dual partition. Write $\psi_{\lambda}$ for the irreducible representation of $S_n$ corresponding to $\lambda$, and $\mathcal{O}_{\lambda}$ for the orbit of nilpotent elements whose Jordan normal form is indexed by $\lambda$. The Springer correspondence \eqref{actual corresp} then maps
\[\psi_{\lambda'}\longmapsto IC_{\mathcal{O}_{\lambda}}(\underline{\mathbb{C}}_{\mathcal{O}_{\lambda}}).\]
See \cite[Section 1.5.16]{yun:17}, for example, for details.
\end{remark}

%
%
%
%
%
%
%
\section{Applications to the universal family of Hessenberg varieties}
\label{what appears}
We apply the tools of the previous section to the universal family of Hessenberg varieties. Consider once again the morphism
\[\mu_H:\, G\times^BH\longrightarrow\mathfrak{g},\]
which is proper by Remark \ref{properness}. The decomposition theorem gives
\begin{equation*}
\mu_{H*}\underline{\mathbb{C}}_{G\times^BH}[-]\cong\bigoplus_{S,b} IC_S\left(\mathcal{L}_{S,b}\right)[-b].
\end{equation*}

The generic fiber of $\mu_H$ has dimension $l=\dim(H/\mathfrak{b})$, and the restriction of $\mu_H$ to the regular semisimple locus $G\times^BH^{\text{rs}}$ is a smooth morphism \cite[Theorem 6]{dem.pro.sha:92}. The decomposition theorem implies that there is an isomorphism in the derived category
\begin{equation}
\label{ics}
\mu_{H*}\underline{\mathbb{C}}_{G\times^BH^{\text{rs}}}[-]\cong\bigoplus_{b=-l}^l \mathcal{H}_b[\dim \mathfrak{g}-b],
\end{equation}
where $\mathcal{H}_b$ is the semisimple local system on $\mathfrak{g}^{\text{rs}}$ whose fiber at $s\in\mathfrak{g}^{\text{rs}}$ is the singular cohomology group $H^{b+l}(\Hess(s,H)).$

\begin{proposition}
\label{appears}
Each local system $\mathcal{H}_{b}$ is a direct sum 
\[\mathcal{H}_b=\bigoplus\mathcal{L}_\psi^{\oplus m_{\psi,b}}\]
of local systems which appear in the decomposition \eqref{springer reps} corresponding to the Grothendieck-Springer resolution.
\end{proposition}
\begin{proof}
The local system $\mathcal{H}_b$ decomposes as a sum of simple local systems
\[\mathcal{H}_{b}=\bigoplus \mathcal{M}_j,\]
with fibers $M_{j,s}$ which are irreducible representations of $\pi_1(\mathfrak{g}^{\text{rs}},s)$. There is a $\pi_1(\mathfrak{g}^{\text{rs}},s)$-equivariant inclusion
\[M_{j,s}\hooklongrightarrow H^{b+l}(\Hess(s,H)).\]
By Proposition \ref{monodromy factors} the action of $\pi_1(\mathfrak{g}^{\text{rs}},s)$ on $H^{b+l}(\Hess(s,H))$ factors through \eqref{factors}. It follows that $\pi_1(\mathfrak{g}^{\text{rs}},s)$ acts on $M_{j,s}$ through an irreducible $W$-representation $\psi$, and therefore $\mathcal{M}_j\cong\mathcal{L}_\psi$.
\end{proof}

We will apply the Fourier transform to the vector bundle $G\times^BH\longrightarrow G/B$. Let $H^\perp\subset\mathfrak{g}$ be the annihilator of $H$ under the Killing form. We have a commutative diagram
\begin{equation*}
\label{two sides}
\begin{tikzcd}[row sep=huge]
G\times^BH		\arrow[hookrightarrow]{r}\arrow{d}{\mu_H}		&G/B\times\mathfrak{g}	\arrow{d}{\mu}\arrow[hookleftarrow]{r} 	&G\times^BH^\perp	\arrow{d}{\mu_{H^\perp}}\\
\mathfrak{g}  		\arrow[equal]{r}								&\mathfrak{g}			\arrow[hookleftarrow]{r}				&\overline{\mathcal{O}_H}.
\end{tikzcd}
\end{equation*}
Since the Hessenberg subspace $H$ contains the fixed positive Borel $\mathfrak{b}$, $H^\perp$ is contained in the nilradical $\mathfrak{n}$. Therefore the image of the morphism $\mu_{H^\perp}$ is contained in the nilpotent cone. This image is irreducible and $G$-stable, and it is closed because $\mu_{H^\perp}$ is proper by Remark \ref{properness}. Therefore it is the closure of a single nilpotent $G$-orbit $\mathcal{O}_H\subset\mathcal{N}$.

\begin{remark}
Suppose that $H\neq\mathfrak{b}$. In this case the orbit $\mathcal{O}_H$ is always non-regular. The vector bundle $G\times^BH^\perp$ is a subbundle of $G\times^B\mathfrak{n}$, and there is a commutative diagram
\begin{equation*}
\begin{tikzcd}[row sep=huge]
G\times^BH^\perp		\arrow[hookrightarrow]{r}\arrow{d}{\mu_{H^\perp}}	&G\times^B\mathfrak{n}=\tilde{\mathcal{N}}	\arrow{d}{\mu_{\mathfrak{n}}}\\
\overline{\mathcal{O}_H}	\arrow[hookrightarrow]{r}						&\mathcal{N}.
\end{tikzcd}
\end{equation*}
This diagram is not Cartesian, and $\mu_{H^\perp}$ always fails to be semismall at $0$ because
\[\codim_{\overline{\mathcal{O}_H}}\{0\}=\dim\mathcal{O}_H<\dim\mathcal{N}=2\dim\mathcal{B}=2\dim\mu_{H^\perp}^{-1}(0).\]
\end{remark}

Recall that the Springer correspondence \eqref{actual corresp} assigns to each irreducible representation $\psi$ of $W$ an irreducible nilpotent orbital complex $IC_{\mathcal{O}_\psi}(\mathcal{M}_\psi)$. We use the Fourier transform to give a necessary condition for an irreducible representation $\psi$ to appear in the dot action.

\begin{theorem}
\label{summands}
Let $s\in\mathfrak{t}$ be a regular element. Suppose that the irreducible $W$-representation $\psi$ appears as a subrepresentation of $H^*(\Hess(s,H))$ under the dot action. Then the intersection $\mathcal{O}_\psi\cap H^\perp$ is nonempty.
\end{theorem}
\begin{proof}
Suppose that $\psi$ is an irreducible summand of $H^*(\Hess(s,H))$ under the dot action. By Corollary \ref{coincidence}, this action coincides with the monodromy action of $W$. It follows that the local system $\mathcal{L}_\psi$ is a direct summand of $\mathcal{H}_b$ for some index $b$. 

The Fourier transform gives an identification
\[\mathfrak{F}\left(\mu_{H*}\underline{\mathbb{C}}_{G\times^BH}[-]\right)\cong\mu_{H^\perp*}\underline{\mathbb{C}}_{G\times^BH^\perp}[-],\]
so the complex $\mathfrak{F}(IC_{\mathfrak{g}^{\text{rs}}}(\mathcal{L}_\psi))=IC_{\mathcal{O}_\psi}(\mathcal{M}_\psi)$ is a direct summand of $\mu_{H^\perp*}\underline{\mathbb{C}}_{G\times^BH^\perp}[-]$. Therefore its support is contained in $\overline{\mathcal{O}_H}=G\cdot H^\perp$. Equivalently, $\mathcal{O}_\psi\cap H^\perp$ is nonempty.
\end{proof}

In view of Remark \ref{type A springer}, Theorem \ref{summands} has the following corollary in type $A$. 

\begin{corollary}
\label{type-a-cor}
Suppose that $G=SL_n$. Let $\lambda$ be a partition of $n$, and suppose that $\psi_{\lambda}$ appears as a subrepresentation of $H^*(\Hess(s,H))$. Then $H^\perp$ contains a nilpotent element whose Jordan normal form corresponds to $\lambda'$.
\end{corollary}

\begin{remark}
Though we were unable to find a reference in the existing literature, Corollary \ref{type-a-cor} can also be proved combinatorially using \cite[Theorem 129]{bro.cho:18} and \cite[Theorem 4]{gas:96}, and appears to be known to experts. Moreover, Martha Precup has explained to us that its converse follows from her recent work with Ji \cite[Lemma 6.1]{pre.ji:19} and Gasharov's Schur expansion of the chromatic symmetric function \cite{gas:96}, which is refined in \cite[Section 6]{sha.wac:16}.
\end{remark}

In type $A$ we can also use the Fourier transform to prove a support theorem for the universal Hessenberg family
\[\mu_H:G\times^BH\longrightarrow\mathfrak{g}.\]

\begin{theorem}
\label{main}
Suppose that $G=SL_n$. There is an isomorphism
\[\mu_{H*}\underline{\mathbb{C}}_{G\times^BH}[-]\cong\bigoplus IC_{\mathfrak{g}^{\text{\emph{rs}}}}(\mathcal{H}_b)[-b]\]
in the derived category, where the local systems $\mathcal{H}_b$ are as defined in \eqref{ics}. In particular, every irreducible summand of $\mu_{H*}\underline{\mathbb{C}}_{G\times^BH}[-]$ has full support.
\end{theorem}
\begin{proof}
Consider once again the $G$-equivariant proper morphism
\[\mu_{H^\perp}:G\times^BH^\perp\longrightarrow\overline{\mathcal{O}_H}.\]
As in Remark \ref{type A springer}, the actions of $G$ on $G\times^B H^\perp$ and on $\overline{\mathcal{O}_H}$ extend to actions of $GL_n$. Let $\mathcal{O}$ be a nilpotent orbit with basepoint $e$. Since centralizers in $GL_n$ are connected, the monodromy action of $\pi_1(\mathcal{O})\cong\pi_0(G_e)$ on $H^*\left(\mu_{H^\perp}^{-1}(e)\right)$ is trivial. 

It follows that each irreducible orbital complex appearing as a summand of $\mu_{H^\perp*}\underline{\mathbb{C}}_{G\times^BH^\perp}[-]$ is trivial. We obtain a decomposition
\[\mu_{H^\perp*}\underline{\mathbb{C}}_{G\times^BH^\perp}[-]\cong\bigoplus_{\mathcal{O},b} IC_{\mathcal{O}}\left(\underline{\mathbb{C}}_\mathcal{O}\right)^{\oplus m_{\mathcal{O},b}}[-b].\]
Applying the Fourier transform, this gives
\begin{equation*}
\mu_{H*}\underline{\mathbb{C}}_{G\times^BH}[-]\cong\bigoplus_{\mathcal{O},b}\mathfrak{F}(IC_O(\underline{\mathbb{C}}_\mathcal{O})[-b])^{\oplus m_{\mathcal{O},b}}.
\end{equation*}

By Remark \ref{type A springer}, every trivial orbital complex $IC_\mathcal{O}(\underline{\mathbb{C}}_\mathcal{O})$ appears in the Springer correspondence. Therefore the Fourier transform of any such complex has full support. The theorem now follows from \eqref{ics}. 
\end{proof}

\begin{remark}
Theorem \ref{main} relates the irreducible summands appearing in the dot action representation to the cohomology of the fibers of the map $\mu_{H^\perp}$, which are a class of ad-nilpotent Hessenberg varieties. This connection is studied in recent work of Precup and Sommers \cite{pre.som:22}, where Theorem \ref{main} is extended to all Lie types.\\
\end{remark}

%
%
%
%
%
%
%
%
%
\section{The local invariant cycle theorem and the K\"ahler package}
\label{locinv}
Suppose that $\varphi:X\longrightarrow Y$ is a proper surjective morphism between smooth algebraic varieties, and let $U\subset Y$ be an open dense subvariety so that $\varphi$ restricts to a smooth morphism along the preimage of $U$. Fix $y\in Y$, and let $D_y\subset Y$ be a sufficiently small Euclidean ball around $y$ such that the restriction
\begin{equation}
\label{linv1}
H^*(\varphi^{-1}(D_y))\longrightarrow H^*(\varphi^{-1}(y))
\end{equation}
is an isomorphism. By the global invariant cycle theorem \cite[Theorem 1.2.2]{dec:17}, for any $u\in U\cap D_y$ the restriction map gives a surjection
\begin{equation}
\label{linv2}
H^*(\varphi^{-1}(U\cap D_y))\longrightarrow H^*(\varphi^{-1}(u))^{\pi_1(U\cap D_y,u)}.
\end{equation}
Composing the inverse of \eqref{linv1} with the natural restriction to $H^*(\varphi^{-1}(U\cap D_y))$ and then with \eqref{linv2}, we obtain a homomorphism of algebras
\[\lambda_y: H^*(\varphi^{-1}(y))\longrightarrow H^*(\varphi^{-1}(u))^{\pi_1(U\cap D_y,u)}.\]
We state a version of the local invariant cycle theorem of Beilinson, Bernstein, and Deligne \cite{bei.ber.del:82}, referring once again to \cite[Section 1.4]{dec:17} for details.

\begin{theorem}[Local Invariant Cycle Theorem]
The map $\lambda_y$ is surjective.
\end{theorem}

Suppose that $d=\dim X-\dim Y$. We say that the fibers of $\varphi$ have \emph{palindromic} Betti numbers if for every $y\in Y$,
\[\dim H^k(\varphi^{-1}(y))= \dim H^{2d-k}(\varphi^{-1}(y))\qquad\text{for all }k\in\mathbb{Z}.\]
In \cite{bro.cho:18}, Brosnan and Chow showed that there is a remarkable connection between the local invariant cycle map and palindromicity. We briefly recall their results.

\begin{theorem}\cite[Theorem 92 and Theorem 102]{bro.cho:18}
\label{brocho}
Suppose that $\varphi:X\longrightarrow Y$ is a projective, surjective morphism between smooth algebraic varieties. The following are equivalent:
\begin{enumerate}
\item The fibers of $\varphi$ have palindromic Betti numbers.
\item Every irreducible summand appearing in the decomposition of $\varphi_*\underline{\mathbb{C}}_X$ has full support and is concentrated in a single cohomological degree.
\item For every $y\in Y$, the local invariant cycle map $\lambda_y$ is an isomorphism.
\end{enumerate}
\end{theorem}

\begin{remark}
While the local invariant cycle theorem applies to any proper morphism, the proof of Theorem \ref{brocho} relies on the relative hard Lefschetz property, which only holds when $\varphi$ is projective.
\end{remark}

We will consider the family
\[\mu_H:\, G\times^BH^{\text{r}}\longrightarrow\mathfrak{g}^{\text{r}},\]
of regular Hessenberg varieties, which is the restriction of $\mu_H$ to the regular locus. In \cite{bro.cho:18}, Brosnan and Chow showed that, when $\mathfrak{g}$ is of type $A$, the fibers of this family have palindromic Betti numbers. In \cite{pre:17}, Precup generalized this to all semisimple Lie algebras by using explicit affine pavings of regular Hessenberg varieties. In view of Precup's result, Theorem \ref{brocho} has the following immediate corollary.

\begin{corollary}
The complex $\mu_{H*}\underline{\mathbb{C}}_{G\times^BH^\text{\emph{r}}}$ has no proper supports.
\end{corollary}

Fix a regular element $x\in\mathfrak{g}$ with Jordan decomposition $x=x_{\text{s}}+x_{\text{n}}$. Let $\mathfrak{t}$ now be a Cartan subalgebra containing $x_{\text{s}}$, and let $W$ be the associated Weyl group. We will prove the following consequence of Theorem \ref{brocho}, which is a straightforward generalization of Theorem 127 of \cite{bro.cho:18}. In the case when $x$ is a regular nilpotent element, this result was proved in \cite{abe.hor.mas.mur.sat:19} using the combinatorics of hyperplane arrangements.

\begin{proposition}
\label{isomorphism}
There is a regular element $s\in\mathfrak{t}$ such that
\begin{equation}
\label{linv isom}
H^*(\Hess(x,H))\cong H^*(\Hess(s,H))^{W_{x_{\text{\emph{s}}}}}
\end{equation}
as graded algebras, where the right-hand side is equipped with the dot action and $W_{x_{\text{\emph{s}}}}$ is the stabilizer of $x_{\text{\emph{s}}}$ in $W$.
\end{proposition}
\begin{proof}
Fix a small Euclidean ball $D_x\subset\mathfrak{g}^{\text{r}}$ centered at $x$ such that the restriction 
\[H^*(\mu_H^{-1}(D_x))\longrightarrow H^*(\Hess(x,H))\]
given by \eqref{linv1} is an isomorphism. By conjugating $x$ if necessary, we can assume that there exists some $s\in\mathfrak{t}^{\text{r}}\cap D_x$. 

Since regular Hessenberg varieties have palindromic Betti numbers \cite{pre:17}, Theorem \ref{brocho} implies that the local invariant cycle theorem gives an isomorphism
\[H^*(\Hess(x,H))\longrightarrow H^*(\Hess(s,H))^{\pi_1(\mathfrak{g}^{\text{rs}}\cap D_x,s)}.\]
We are then reduced to proving the following fact: the image of the composition
\begin{equation}
\label{local monodromy factors}
\pi_1(\mathfrak{g}^{\text{rs}}\cap D_x,s)\longrightarrow \pi_1(\mathfrak{g}^{\text{rs}},s)\xlongrightarrow{\eqref{factors}} W
\end{equation}
is a subgroup of $W$ conjugate to $W_{x_{\text{s}}}.$

Let $\mathfrak{b}$ now be a Borel subalgebra containing $x$ and $\mathfrak{t}$, and let $B$ be the corresponding Borel subgroup. Restricting the Grothendieck--Springer resolution to the regular locus, we obtain a $W$-equivariant Cartesian diagram
\begin{equation}
\label{cartesian}
\begin{tikzcd}[row sep=huge]
\tilde{\mathfrak{g}}^{\text{r}}=G\times^B\mathfrak{b}^{\text{r}}		\arrow{r}\arrow{d}{\mu_{\mathfrak{b}}}	&\mathfrak{t}	\arrow{d}\\
\mathfrak{g}^{\text{r}}	\arrow{r}					&\mathfrak{t}/W.
\end{tikzcd}
\end{equation}
The top horizontal arrow maps the point $[1:x]\in\mu_\mathfrak{b}^{-1}(x)$ to the semisimple part $x_{\text{s}}$. 

By \cite[Proposition 106]{bro.cho:18}, the image of the composition \eqref{local monodromy factors} is a subgroup of $W$ conjugate to the stabilizer of $[1:x]$ under the action of $W$ on $\tilde{\mathfrak{g}}^{\text{r}}$. Because \eqref{cartesian} is Cartesian, the stabilizer of $[1:x]$ in $W$ is precisely $W_{x_{\text{s}}}$. \end{proof}

We will use Proposition \ref{isomorphism} to show that $H^*(\Hess(x,H))$ has the ``K\"ahler'' package. This generalizes Proposition 8.14 and Theorem 12.1 of \cite{abe.hor.mas.mur.sat:19}, which apply in the case when $x$ is regular nilpotent. Compared to these results, our proofs are simplified by the identification of the dot action with the monodromy.

\begin{theorem}
\label{kahler}
Let $x\in\mathfrak{g}$ be a regular element and let $l$ be the dimension of the regular Hessenberg variety $\Hess(x,H)$.
\begin{enumerate}
\item \emph{(}Poincar\'{e} duality\emph{)} The cohomology ring $H^*(\Hess(x,H))$ is a Poincar\'{e} duality algebra.
\end{enumerate}
Moreover, there is a nonzero ``K\"ahler'' class $\omega\in H^2(\Hess(x,H))$ satisfying the following properties: 
\begin{enumerate}
\setcounter{enumi}{1}
\item \emph{(}Hard Lefschetz\emph{)} Multiplication by $\omega^k$ induces an isomorphism
\[H^{l-k}(\Hess(x,H))\xlongrightarrow{\omega^k}H^{l+k}(\Hess(x,H))\qquad\text{for every }0\leq k\leq l.\]
\item \emph{(}Hodge--Riemann\emph{)} For every $0\leq k\leq l$, the symmetric bilinear form
\begin{align*}
H^k(\Hess(x,H))\times H^k(\Hess(x,H))&\longrightarrow\mathbb{C} \\
					(\alpha,\beta)&\longmapsto (-1)^k\int\alpha\cup\beta\cup\omega^{l-k}
\end{align*}
is positive-definite on the kernel of the linear map
\begin{align*}
H^k(\Hess(x,H))\xlongrightarrow{\,\,\omega^{l-k+1}\,\,} H^{2l-k+2}(\Hess(x,H)).
\end{align*}
\end{enumerate}
\end{theorem}
\begin{proof}
(1) Let $s\in\mathfrak{t}$ be the regular semisimple element of \eqref{linv isom}. Because $\Hess(s,H)$ is smooth, there is a nondegenerate Poincar\'{e} duality pairing
\[H^k(\Hess(s,H))\times H^{2l-k}(\Hess(s,H))\longrightarrow H^{2l}(\Hess(s,H))\longrightarrow\mathbb{C}.\]
The first arrow is given by the cup product, and the second is given by evaluating on the homology class which is the sum of the fundamental classes of the irreducible components of $\Hess(s,H)$.

The monodromy action of $W$ preserves both the cup product and this sum, so this pairing is $W$-invariant. Restricting it to 
\[H^k(\Hess(s,H))^{W_{x_{\text{s}}}}\times H^{2l-k}(\Hess(s,H))^{W_{x_{\text{s}}}}\longrightarrow H^{2l}(\Hess(s,H))^{W_{x_{\text{s}}}}\longrightarrow\mathbb{C}\]
gives a nondegenerate pairing on $H^*(\Hess(x, H))\cong H^*(\Hess(s,H))^{W_{x_{\text{s}}}}.$

(2 and 3) The usual inclusions form a commutative diagram
\begin{equation*}
\begin{tikzcd}[row sep=huge]
G\times^BH		\arrow{d}						&G\times^BH^{\text{rs}}	\arrow[l, hook']\\
G/B											&\Hess(s,H) 			\arrow[u, hook]\arrow[l, hook'],
\end{tikzcd}
\end{equation*}
where the left vertical arrow is the bundle map.

This gives a commutative diagram of pullbacks in cohomology
\begin{equation*}
\begin{tikzcd}
H^*(G\times^BH)		\arrow{r}			&H^*(G\times^BH^{\text{rs}})	\arrow{dd}\\
&&\\
H^*(G/B)				\arrow{r}\arrow{uu}					&H^*(\Hess(s,H)).
\end{tikzcd}
\end{equation*}
Since $G\times^BH^{\text{rs}}\longrightarrow\mathfrak{g}^{\text{rs}}$ is a locally trivial fibration, the image of the right vertical arrow lies in the space $H^*(\Hess(s,H))^W$ of monodromy invariants. It follows that the image of the restriction
\[H^*(G/B)\longrightarrow H^*(\Hess(s,H))\]
also lies in $H^*(\Hess(s,H))^W$. 

Since both $G/B$ and $\Hess(s,H)$ are smooth projective varieties, there is a K\"ahler class in $H^2(G/B)$ whose image is a K\"ahler class $\omega\in H^2(\Hess(s,H))$. By the discussion above, $\omega$ is $W$-invariant. Its preimage under \eqref{linv isom} then satisfies (2) and (3) for $H^*(\Hess(x,H))$. 
\end{proof}

The hard Lefschetz property implies that the even-degree Betti numbers of any regular Hessenberg variety are unimodal \cite[Corollary 1.2.9]{max:19}. This unimodality has already been observed in the regular nilpotent case by Tymoczko \cite[Section 9.3]{tym:06}. In type $A$, it also follows from \cite[Proposition 10.2]{sha.wac:16} and \cite[Theorem 129]{bro.cho:18}.

\begin{corollary}
Let $x\in\mathfrak{g}$ be a regular element. The even Betti numbers of the Hessenberg variety $\Hess(x,H)$ form a unimodal sequence.\\
\end{corollary}

%
%
%
%
%
%
%
%
%
\appendix
\section*{Appendix: Monodromy actions on equivariant cohomology}
\setcounter{section}{1}
\setcounter{equation}{0}

In this appendix we recall some fundamental background on equivariant cohomology, and we use it to prove a series of more specialized facts about monodromy actions in the equivariant setting. These results, which are straightforward but which we were unable to easily find  in the literature, are needed to show in Corollary \ref{coincidence} that the Tymoczko dot action is induced by monodromy. 

\subsection{Background on equivariant cohomology}
\label{Background on equivariant cohomology}
First we recall some standard facts about equivariant cohomology. We refer to \cite[Section 1]{bri:98} for more details. Let $G$ be a connected Lie group and fix a principal $G$-bundle $EG\longrightarrow BG$ whose total space is contractible. Suppose that $X$ is a $G$-space---that is, a topological space equipped with a continuous action of $G$. Then the product $X\times EG$ carries a free diagonal action of $G$. The $G$-equivariant cohomology of $X$ is the singular cohomology of the quotient $X\times^G EG$:
\[H_G^*(X)=H^*(X\times^G EG).\]
It is independent of the choice of $EG\longrightarrow BG$.

A continuous $G$-equivariant map $f:X\longrightarrow Y$ between $G$-spaces $X$ and $Y$ induces a pullback 
\begin{equation}
\label{Equation: Functoriality} 
H_G^*(Y)=H^*(Y\times^G EG)\longrightarrow H^*(X\times^G EG)=H_G^*(X),
\end{equation}
which is a homomorphism of graded algebras. In particular, the map $X\longrightarrow \{\pt\}$ induces
\[H_G^*(\text{pt})\longrightarrow H_G^*(X),\]
giving $H_G^*(X)$ the structure of a graded $H_G^*(\pt)$-algebra. The pullback \eqref{Equation: Functoriality} is a homomorphism of graded $H_G^*(\text{pt})$-algebras. 

Let $K$ be any closed, connected subgroup of $G$. Then $EG\longrightarrow EG/K$ is a principal $K$-bundle whose total space is contractible. Therefore 
\[H_K^*(X)\cong H^*(X\times^K EG).\]
The pullback along $X\times^K EG\longrightarrow X\times^G EG$ induces a homomorphism of graded algebras
\[H_G^*(X)\longrightarrow H_K^*(X),\]
called a \emph{specialization map}. In particular, when $K=\{1\}$ is the trivial subgroup of $G$, this gives a natural map
\[H_G^*(X)\longrightarrow H^*(X).\]

\begin{lemma}
\label{real forms}
Suppose that $K$ is a maximal compact subgroup of $G$. Then specialization induces an isomorphism
\[H^*_G(X)\cong H^*_K(X).\]
\end{lemma}
\begin{proof}
The fibration
\[X\times^KEG\longrightarrow X\times^GEG\]
has fiber isomorphic to $G/K$. Since $K$ is a maximal compact subgroup of $G$, the quotient $G/K$ is contractible \cite[Theorem 6]{iwa:49}. It follows that the pullback
\[H^*(X\times^GEG)\longrightarrow H^*(X\times^KEG)\]
is an isomorphism. 
\end{proof}

\subsection{Equivariant cohomology in local systems}
Let $G$ be a semisimple complex algebraic group and let $\mathfrak{g}$ be its Lie algebra. Fix a maximal torus $T$ with Lie algebra $\mathfrak{t}$ and Weyl group $W$. Suppose that $X$ is a smooth algebraic $G$-variety equipped with a smooth, $G$-equivariant, proper surjective morphism $\varphi:X\longrightarrow\mathfrak{g}^{\text{rs}}$. For any subset $V\subset\mathfrak{g}$, let
\[X_V=\varphi^{-1}(V).\]
In particular, $X_s$ is the fiber above a regular semisimple element $s\in\mathfrak{g}$, and $X_{\mathfrak{t}}$ is the restriction of $X$ to $\mathfrak{t}^{\text{r}}$.

\begin{theorem}
\label{actual action}
Let $s\in\mathfrak{g}$ be a regular semisimple element. For any non-negative integer $k$, there exists a local system $\mathcal{L}^k$ on $\mathfrak{g}^{\text{\emph{rs}}}$ whose fiber at $s$ is $H^k_{G_s}(X_s).$ Therefore the fundamental group $\pi_1(\mathfrak{g}^{\text{\emph{rs}}},s)$ acts on the equivariant cohomology ring $H^*_{G_s}(X_s)$ by monodromy.
\end{theorem}

Theorem \ref{actual action} will follow from the next Proposition.

\begin{proposition}
\label{torus equivariant}
For any non-negative integer $k$, there exists a local system on $\mathfrak{t}^{\text{\emph{r}}}$ whose fiber at $s\in\mathfrak{t}^{\text{\emph{r}}}$ is $H^k_{T}(X_s).$
\end{proposition}
\begin{proof}
Let $K$ be a maximal compact subgroup of $G$ such that $S=K\cap T$ is a maximal compact torus in $T$. By Lemma \ref{real forms}, 
\[H^k_{T}(X_s)=H^k_{S}(X_s).\]

Fix $n\geq k$ and $n\geq 2\dim X$, and let $V$ be a $n$-connected, compact manifold on which $K$ acts freely. (Such a manifold always exists, see for example \cite[Theorem 19.6]{ste:57}.) Then
\[H^k_{S}(X_s)=H^k(X_s\times^{S} V).\]
(See \cite[Section 1]{bri:98}.) Since $\varphi:X_\mathfrak{t}\longrightarrow\mathfrak{t}^{\text{r}}$ is smooth and proper, the induced map
\[\Phi:X_\mathfrak{t}\times^SV\longrightarrow \mathfrak{t}^{\text{r}}\]
is a proper locally trivial fibration. The $k$-th derived pushforward $R^k\Phi_*\underline{\mathbb{C}}_{X_\mathfrak{t}\times^SV}$ is therefore a local system. By proper base change 
\begin{align*}
(R^k\Phi_*\underline{\mathbb{C}}_{X_\mathfrak{t}\times^SV})_{s}	& \cong H^k(\Phi^{-1}(s)) \\
																	& \cong H^k(X_s\times^{S}V) \cong H^k_S(X_s)  \cong H^k_{T}(X_s).\qedhere
																\end{align*}

\end{proof}

\begin{proof}[Proof of Theorem \ref{actual action}]
We keep the same notation as in the proof of Proposition \ref{torus equivariant}. There is a Cartesian diagram 
\begin{equation}
\label{torus diagram}
\begin{tikzcd}[row sep=huge]
X_{\mathfrak{t}}\times^{S}V	\arrow{r}\arrow{d}{\Phi} 				&X_{\mathfrak{t}}\times^{N_K(S)}V	\arrow{d}{\bar{\Phi}}\\
\mathfrak{t}^{\text{r}}			\arrow{r}							&\mathfrak{t}^{\text{r}}/W,
\end{tikzcd}
\end{equation}
where $N_K(S)$ is the normalizer of the compact torus $S$ in $K$. All the maps above are smooth, locally trivial fibrations, and $R^k\Phi_*\underline{\mathbb{C}}_{X_\mathfrak{t}\times^SV}$ is the pullback of the local system 
\begin{equation}
\label{ellbar}
\bar{\mathcal{L}}^k=R^k\bar{\Phi}_*\underline{\mathbb{C}}_{X_\mathfrak{t}\times^{N_K(S)}V}.
\end{equation}
The Chevalley isomorphism gives a smooth surjection of algebraic varieties
\[\chi:\mathfrak{g}^{\text{rs}}\longrightarrow \mathfrak{g}^{\text{rs}}/G\cong\mathfrak{t}^{\text{r}}/W,\]
and the fiber at $s\in\mathfrak{g}^{\text{rs}}$ of the pullback 
\[\mathcal{L}^k=\chi^*\bar{\mathcal{L}}^k\]
is exactly $H^k_{G_s}(X_s).$ Therefore there is a monodromy action of $\pi_1(\mathfrak{g}^{\text{rs}},s)$ on $H^k_{G_s}(X_s)$.
\end{proof}

\begin{remark}
\label{same pi-1-rk}
Consider once again the Cartesian diagram \eqref{rs-groth-springer}
\begin{equation*}
\begin{tikzcd}[row sep=huge]
\tilde{\mathfrak{g}}^{\text{rs}}	\arrow{r}\arrow{d}{\mu_{\mathfrak{b}}} 			&\mathfrak{t}^{\text{r}}	\arrow{d}{}\\
\mathfrak{g}^{\text{rs}}		\arrow{r}{\chi}								&\mathfrak{t}^{\text{r}}/W.
\end{tikzcd}
\end{equation*}
For any regular element $s\in\mathfrak{t}$, let $\bar{s}\in\mathfrak{t}^{\text{r}}/W$ be its image under right vertical arrow and let $\tilde{s}$ be a fixed preimage under the top horizontal arrow. 

The adjoint quotient $\chi$ is a smooth morphism with fiber isomorphic to the variety $G/T$, which is connected and simply-connected. Taking the appropriate long exact sequence of homotopy groups gives group isomorphisms
\begin{equation}
\label{same pi-1-eq}
\pi_1(\mathfrak{g}^{\text{rs}},s)\cong\pi_1(\mathfrak{t}^{\text{r}}/W,\bar{s})\qquad\text{and}\qquad\pi_1(\tilde{\mathfrak{g}}^{\text{rs}},\tilde{s})\cong\pi_1(\mathfrak{t}^{\text{r}},s).
\end{equation}

The vertical arrows are Galois covers with Galois group $W$, and in view of \eqref{same pi-1-eq} they induce two isomorphic short exact sequences of groups:
\begin{align}
\label{fundamentalW}
&1\longrightarrow\pi_1(\tilde{\mathfrak{g}}^{\text{rs}},\tilde{s})\longrightarrow\pi_1(\mathfrak{g}^{\text{rs}},s)\xlongrightarrow{(\star)} W\longrightarrow 1;\\
\label{fundamentalWtorus}
&1\longrightarrow\pi_1(\mathfrak{t}^{\text{r}},s)\longrightarrow\pi_1(\mathfrak{t}^{\text{r}}/W,\bar{s})\xlongrightarrow{(\bullet)} W\longrightarrow 1. \nonumber
\end{align}
The surjective map of \eqref{fundamentalW} is the homomorphism \eqref{factors} of Section \ref{monodromy section}. 

In particular, these short exact sequences imply that the monodromy action of $\pi_1(\tilde{\mathfrak{g}}^{\text{rs}},s)$ given by Theorem \ref{actual action} coincides with the monodromy action of $\pi_1(\mathfrak{t}^{\text{r}}/W)$ on the fibers of the local system $\bar{\mathcal{L}}^k$ defined in \eqref{ellbar}.
\end{remark}

\begin{remark}
\label{weyl action}
(1) The natural action of $W$ on $EG/T$ induces an action of $W$ on the cohomology ring $H^*_T(\pt)=H^*(EG/T)$. The natural isomorphism $H^*_T(\pt)\cong\mathbb{C}[\mathfrak{t}]$ \cite[Example 1.2]{bri:98} is equivariant with respect to this action.

(2) Let $s\in\mathfrak{t}$ be a regular element. Because $T$ acts trivially on $\mathcal{B}_s=\mu_{\mathfrak{b}}^{-1}(s)$, the K\"unneth theorem gives an isomorphism 
\[H_T^*(\mathcal{B}_s)=H^*(\mathcal{B}_s\times^TEG)=H^*(\mathcal{B}^s) \otimes H^*(EG/T)=H^0(\mathcal{B}_s)\otimes H^*_T(\pt).\]
The action of the Weyl group on $H^0(\mathcal{B}_s)$ and on $H^*_T(\pt)$ gives a diagonal $W$-action on $H_T^*(\mathcal{B}_s)$.
\end{remark}

We will show that Remark \ref{weyl action} gives concrete descriptions of the monodromy action defined in Theorem \ref{actual action} in the case $X=\mathfrak{g}^{\text{rs}}$ and $X=\tilde{\mathfrak{g}}^\text{rs}$.

\begin{proposition}
\label{point}
Let $X=\mathfrak{g}^{\text{\emph{rs}}}$, let $\varphi$ be the identity map, and let $s\in\mathfrak{t}$ be a regular element. The monodromy action of $\pi_1(\mathfrak{g}^{\text{\emph{rs}}},s)$ on $H^*_T(s)$ factors through \eqref{factors} and agrees with the $W$-action defined in Remark \ref{weyl action}.
\end{proposition}
\begin{proof}
In view of Remark \ref{same pi-1-rk}, it is sufficient to prove the statement for the monodromy action of $\pi_1(\mathfrak{t}^{\text{r}}/W,\bar{s})$ on the fiber of the local system $\bar{\mathcal{L}}^k$. Consider diagram \eqref{torus diagram} in this case:
\begin{equation*}
\begin{tikzcd}[row sep=huge]
\mathfrak{t}^{\text{r}}\times^{S}V	\arrow{r}\arrow{d}{\Phi} 				&\mathfrak{t}^{\text{r}}\times^{N_K(S)}V	\arrow{d}{\bar{\Phi}}\\
\mathfrak{t}^{\text{r}}			\arrow{r}							&\mathfrak{t}^{\text{r}}/W.
\end{tikzcd}
\end{equation*}
Since the maximal compact torus $S$ acts trivially on $\mathfrak{t}$, the left vertical arrow is a trivial fibration. It follows that the action of the monodromy factors through ($\bullet$). 

The resulting action of $W$ on $H^k(\bar{\Phi}^{-1}(\bar{s}))$ comes from the natural $W$-action on the fiber
\[\Phi^{-1}(s)=\{s\}\times^S V\cong V/S.\]
It follows that the monodromy action is precisely the action defined in Remark \ref{weyl action}(1).
\end{proof}

\begin{corollary}
\label{factorscor}
Let $X=\tilde{\mathfrak{g}}^{\text{\emph{rs}}}$, let $\varphi=\mu_{\mathfrak{b}}$ be the Grothendieck--Springer morphism, and let $s\in\mathfrak{t}$ be a regular element. The monodromy action of $\pi_1(\mathfrak{g}^{\text{\emph{rs}}},s)$ on $H_{T}^*(\mathcal{B}_s)$ factors through \eqref{factors} and agrees with the $W$-action defined in Remark \ref{weyl action}.
\end{corollary}
\begin{proof}
Again, consider diagram \eqref{torus diagram} in this case:
\begin{equation*}
\begin{tikzcd}[row sep=huge]
\tilde{\mathfrak{t}}^{\text{r}}\times^{S}V	\arrow{r}\arrow{d}{\Phi} 				&\tilde{\mathfrak{t}}^{\text{r}}\times^{N_K(S)}V	\arrow{d}{\bar{\Phi}}\\
\mathfrak{t}^{\text{r}}			\arrow{r}							&\mathfrak{t}^{\text{r}}/W.
\end{tikzcd}
\end{equation*}
The action of $W$ on $H^*(\bar{\Phi}^{-1}(\bar{s}))$ is induced by the action of $W$ on the fiber $\Phi^{-1}(s)=\mathcal{B}_s\times V/S$, which is the diagonal action. It follows that the K\"unneth isomorphism 
\[H^*_S(\mathcal{B}_s)=H^*(\mathcal{B}_s)\otimes H^*_S(s)\]
is $W$-equivariant. The statement now follows from Proposition \ref{point}.
\end{proof}

In the next two propositions we show that the monodromy action defined in Theorem \ref{actual action} is compatible with specializations and pullbacks.

\begin{proposition}
\label{equivariant edge}
Let $s\in\mathfrak{g}$ be a regular semisimple element. The specialization map 
\begin{equation}
\label{specmap}
H^*_{G_s}(X_s)\longrightarrow H^*(X_s)
\end{equation}
is equivariant with respect to the monodromy action of $\pi_1(\mathfrak{g}^{\text{\emph{rs}}},s)$.
\end{proposition}
\begin{proof}
It is enough to check this for $s\in\mathfrak{t}^{\text{r}}$. We keep the notation used in the proof of Proposition \ref{torus equivariant}. There is a commutative diagram
\begin{equation*}
\begin{tikzcd}[row sep=huge]
X_{\mathfrak{t}}\times V		\arrow{r}{\alpha}\arrow{rd}[swap]{\varphi}		&X_{\mathfrak{t}}\times^{S}V	\arrow{d}{\Phi} \\																&\mathfrak{t}^{\text{r}},	
\end{tikzcd}
\end{equation*}
where $\alpha$ is the quotient by the diagonal action of $S$.

In degree $k$, the specialization map is given by the fibers of the adjunction
\[R^k\Phi_*\underline{\mathbb{C}}_{X_{\mathfrak{t}}\times^{S}V}\longrightarrow R^k\Phi_{*}\alpha_*\alpha^*\underline{\mathbb{C}}_{X_{\mathfrak{t}}\times^{S}V}=R^k\varphi_{*}\underline{\mathbb{C}}_{X_{\mathfrak{t}}\times V}.\]
Since it is induced by a morphism of local systems, \eqref{specmap} is monodromy-equivariant.
\end{proof}

Now let $Y$ be another smooth $G$-variety equipped with a smooth, $G$-equivariant, proper surjective morphism $\psi:Y\longrightarrow\mathfrak{g}^{\text{rs}}$. Once again write $Y_V=\psi^{-1}(V)$ for the restriction of $Y$ to a subset $V\subset\mathfrak{g}^{\text{rs}}$. Suppose that $f:X\longrightarrow Y$ is a smooth, $G$-equivariant morphism and that $\varphi=\psi\circ f.$

\begin{proposition}
\label{equivariant restriction}
Let $s\in\mathfrak{g}$ be a regular semisimple element. The pullback map
\begin{equation}
\label{restrmap}
H^*_{G_s}(Y_s)\longrightarrow H^*_{G_s}(X_s)
\end{equation}
is equivariant with respect to the monodromy action of $\pi_1(\mathfrak{g}^{\text{\emph{rs}}},s)$.
\end{proposition}
\begin{proof}
Again it is sufficient to consider $s\in\mathfrak{t}^{\text{r}}$. We keep the notation of Proposition \ref{torus equivariant}. There is a commutative diagram
\begin{equation*}
\begin{tikzcd}[row sep=huge]
X_\mathfrak{t}\times^{S}V		\arrow{r}{f}\arrow{rd}[swap]{\Phi}		&Y_\mathfrak{t}\times^{S}V\arrow{d}{\Psi}\\
  														&\mathfrak{g}^{\text{rs}},
\end{tikzcd}
\end{equation*}
where the horizontal arrow is induced by $f$ and $\Psi$ is induced by $\psi$. In degree $k$, the pullback map is given by the fibers of the adjunction
\[R^k\Psi_*\underline{\mathbb{C}}_{Y_\mathfrak{t}\times^{S}V}
					\longrightarrow R^k\Psi_* f_*f^*\underline{\mathbb{C}}_{Y_\mathfrak{t}\times^{S}V}
						=R^k\Psi_* f_*\underline{\mathbb{C}}_{X_\mathfrak{t}\times^{S}V}
						=R^k\Phi_*\underline{\mathbb{C}}_{X_\mathfrak{t}\times^{S}V},\]
Since it is induced by a morphism of local systems, \eqref{restrmap} is monodromy-equivariant.
\end{proof}

We are now ready to prove Propositon \ref{monodromy factors}, which characterizes the monodromy action on the singular cohomology of a regular semisimple Hessenberg variety.

\begin{proof}[Proof of Proposition \ref{monodromy factors}]
As in Theorem \ref{actual action} of the appendix, there is a monodromy action of $\pi_1(\mathfrak{g}^{\text{rs}},s)$ on the $T$-equivariant cohomology of $\mathcal{B}_s=\mu^{-1}_\mathfrak{b}(s)$. By Corollary \ref{factorscor}, this action factors through \eqref{factors}.

By Remark \ref{fibers}, 
\[\Hess(s,H)^T=\mathcal{B}_s.\] 
Since $\Hess(s,H)$ is a GKM variety, Proposition \ref{gkm}(1) implies that the restriction map
\[H^*_T(\Hess(s,H))\longrightarrow H^*_T(\mathcal{B}_s)\]
is injective. It is also $\pi_1(\mathfrak{g}^{\text{rs}},s)$-equivariant by Proposition \ref{equivariant restriction}. It follows that the action of $\pi_1(\mathfrak{g}^{\text{rs}},s)$ on $H^*_T(\Hess(s,H))$ factors through \eqref{factors}. 

At the same time, the specialization map
\[H^*_T(\Hess(s,H))\longrightarrow H^*(\Hess(s,H))\]
is surjective by Proposition \ref{gkm}(2) and $\pi_1(\mathfrak{g}^{\text{rs}},s)$-equivariant by Proposition \ref{equivariant edge}. Therefore the action of $\pi_1(\mathfrak{g}^{\text{rs}},s)$ on $H^*(\Hess(s,H))$ also factors through \eqref{factors}.
\end{proof}

%
%
%
%
%
%
%
%
%
\bibliographystyle{plain}
\bibliography{hess}

\end{document}